\documentclass[a4,11pt]{amsart}
\usepackage{amsfonts}
\usepackage{mathrsfs}
\usepackage{amsthm,amsxtra}
\usepackage{amssymb}
\usepackage{amsmath}
\usepackage{amscd}
\usepackage{enumerate}
\usepackage{comment}

\usepackage{pgfplots}
\pgfplotsset{compat=1.15}
\usepackage{mathrsfs}
\usetikzlibrary{arrows}
\usetikzlibrary{calc}
\usepackage{t1enc}
\usepackage[mathscr]{eucal}
\usepackage{indentfirst}
\usepackage{graphicx}
\graphicspath{ {images/} }
\usepackage[margin=2.9cm]{geometry}
\usepackage{tikz-cd}
\usepackage{tikz}

\theoremstyle{plain}

\newtheorem{theorem}{Theorem}[section]
\newtheorem{proposition}[theorem]{Proposition}
\newtheorem{lemma}[theorem]{Lemma}
\newtheorem{corollary}[theorem]{Corollary}
\newtheorem{thmalpha}{Theorem}

\theoremstyle{definition}

\newtheorem{question}[theorem]{Question}

\theoremstyle{remark}

\DeclareMathOperator{\Cor}{Cor}

\DeclareMathOperator{\Aut}{Aut}

\DeclareMathOperator{\Out}{Out}

\DeclareMathOperator{\C}{C}

\DeclareMathOperator{\PSL}{PSL}

\DeclareMathOperator{\Sz}{Sz}
\DeclareMathOperator{\OO}{\mathcal{O}}
\DeclareMathOperator{\CC}{\mathcal{C}}
\DeclareMathOperator{\PG}{\mathbf{P}\mathbf{G}}

\DeclareMathOperator{\PGL}{PGL}
\DeclareMathOperator{\PGO}{PGO}
\DeclareMathOperator{\AGL}{AGL}
\DeclareMathOperator{\FF}{\mathbb{F}}
\DeclareMathOperator{\PGamma}{P\Gamma L}

\makeatletter
\DeclareFontFamily{OMX}{MnSymbolE}{}
\DeclareSymbolFont{MnLargeSymbols}{OMX}{MnSymbolE}{m}{n}
\SetSymbolFont{MnLargeSymbols}{bold}{OMX}{MnSymbolE}{b}{n}
\DeclareFontShape{OMX}{MnSymbolE}{m}{n}{
    <-6>  MnSymbolE5
   <6-7>  MnSymbolE6
   <7-8>  MnSymbolE7
   <8-9>  MnSymbolE8
   <9-10> MnSymbolE9
  <10-12> MnSymbolE10
  <12->   MnSymbolE12
}{}
\DeclareFontShape{OMX}{MnSymbolE}{b}{n}{
    <-6>  MnSymbolE-Bold5
   <6-7>  MnSymbolE-Bold6
   <7-8>  MnSymbolE-Bold7
   <8-9>  MnSymbolE-Bold8
   <9-10> MnSymbolE-Bold9
  <10-12> MnSymbolE-Bold10
  <12->   MnSymbolE-Bold12
}{}

\let\llangle\@undefined
\let\rrangle\@undefined
\DeclareMathDelimiter{\llangle}{\mathopen}%
                     {MnLargeSymbols}{'164}{MnLargeSymbols}{'164}
\DeclareMathDelimiter{\rrangle}{\mathclose}%
                     {MnLargeSymbols}{'171}{MnLargeSymbols}{'171}
\makeatother

\title{Triples of involutions in $\PGL(2,q)$ and their incidence geometries}

\author{Philippe Tranchida}
\address{Philippe Tranchida, Max Planck Institute for Mathematics in the Sciences, Inselstrasse 22 D-04107 Leipzig, Orcid number 0000-0003-0744-4934.}
\curraddr{}
\email{tranchida.philippe@gmail.com}
\urladdr{}

\date{\today}
{}
\begin{document}

\begin{abstract}

    For $q = p^n$ with $p$ an odd prime, the projective linear group $\PGL(2,q)$ can be seen as the stabilizer of a conic $\OO$ in a projective plane $\pi = \PG(2,q)$. In that setting, involutions of $\PGL(2,q)$ correspond bijectively to points of $\pi$ not in $\OO$. Triples of involutions $\{ \alpha_P,\alpha_Q,\alpha_R \}$ of $\PGL(2,q)$ can then be seen also as triples of points $\{P,Q,R\}$ of $\pi$. 
    We investigate the interplay between algebraic properties of the group $H = \langle \alpha_P,\alpha_Q,\alpha_R \rangle$ generated by three involutions and geometric properties of the triple of points $\{P,Q,R\}$. In particular, we show that the coset geometry $\Gamma = \Gamma(H,(H_0,H_1,H_2))$, where $H_0 = \langle \alpha_Q,\alpha_R \rangle, H_1 = \langle \alpha_P,\alpha_R \rangle$ and $H_2 = \langle \alpha_P,\alpha_Q \rangle$ is a regular hypertope if and only if $\{P,Q,R\}$ is a strongly non self-polar triangle, a terminology we introduce.
    This entirely characterizes hypertopes of rank $3$ with automorphism group a subgroup of $\PGL(2,q)$. 
    As a corollary, we obtain the existence of hypertopes of rank $3$ with non linear diagrams and with automorphism group $\PGL(2,q)$, for any $q = p^n$ with $p$ an odd prime. We also study in more details the case where the triangle $\{P,Q,R\}$ is tangent to $\OO$.
 
\end{abstract}

\maketitle

\section{Introduction}

Realising groups as automorphism groups of geometric structures is a common and ubiquitous theme in mathematics. For the projective linear groups $\PGL(2,q)$, a lot of literature on the subject already exists.

In \cite{cherkassoff1993groups}, Cherkassoff and Sjerve show that $\PGL(2,q)$, for $q$ a power of an odd prime, is generated by three involutions, two of which commute. This implies that $\PGL(2,q)$ is the automorphism group of a regular abstract polyhedra. Indeed, automorphism groups of abstract polytopes correspond bijectively to string C-groups These groups are generated by involutions, some of which must commute, and all of which must satisfy the so-called intersection property. See \cite{mcmullen2002abstract} for more information on abstract polytopes.

On the other hand, Leemans and Schulte ~\cite{Leemans2009Polytope} prove that $\PGL(2,q)$ cannot be the automorphism group of a polytope of rank $5$ or more, and that, if it is of rank $4$, it must necessarily be the $4$-simplex and $q$ must be equal to $5$. In the chiral world, Leemans, Moerenhout and O’Reilly-Regueiro ~\cite{ChiralPolyPGL} investigate whether $\PGL(2,q)$ and $\PSL(2,q)$ can be the automorphism groups of chiral polytopes. In particular, they show that it can only happen for chiral polytopes of rank $4$.

Conder, Potočnik and Širáň ~\cite{RegularHypermapsPGL} study regular hypermaps for projective linear groups. They enumerate all the triples of involutions that generate the group $\PGL(2,q)$. 

Cameron, Omidi and Tayfeh-Rezaie ~\cite{3-design} investigate instead when $\PGL(2,q)$ is the automorphism group of a $3-(q+1,k,l)$ design.

In \cite{HighlySymmetricHypertopes}, Leemans, Fernandes and Weiss introduce a generalization of regular polytopes, that they call regular hypertopes. The main idea is to get rid of the condition that forces an abstract polytope to have a linear diagram. More precisely, regular hypertopes are defined as thin, residually connected and flag-transitive incidence geometries. Their automorphism group are also generated by involutions, with no commuting conditions, and also satisfy the intersection property. Such groups are called C-groups. Contrary to the situation for polytopes, it is not always possible to reconstruct a regular hypertope from any given C-group.
Indeed, the flag-transitivity of the associated hypertope is not guaranteed (see \cite[Example 4.4]{HighlySymmetricHypertopes}, for example).

In this paper, we give a sufficient and necessary condition for a subgroup $H$, generated by involutions, of $\PGL(2,q)$ to be the automorphism group of a regular hypertope of rank $3$. We now summarize the main results of the paper.

Throughout this article, we will consider the group $\PGL(2,q)$ as the stabilizer of a conic $\OO$ in a projective plane $\pi = \PG(2,q)$. In this action, every involutions of $\PGL(2,q)$ is a perspectivity of $\pi$. Moreover, for every point $X \in \pi - \OO$, there in a unique involution $\alpha_X$ of $\PGL(2,q)$ that has $X$ as its center. Let $\alpha_P,\alpha_Q$ and $\alpha_R$ be three involutions of $\PGL(2,q)$, with centers $P,Q$ and $R$ respectively. Let $H$ be the subgroup of $\PGL(2,q)$ generated by $\alpha_P,\alpha_R$ and $\alpha_Q$ and let $H_0 = \langle \alpha_Q,\alpha_R \rangle, H_1 = \langle \alpha_P \alpha_R \rangle$ and $H_2 = \langle \alpha_P, \alpha_R \rangle$. We can then consider the coset geometry $\Gamma = \Gamma(H,(H_0,H_1,H_2))$.

Our first main Theorem completely characterizes in geometry terms when this geometry $\Gamma$ is an hypertope. For the definition of strong self-polarity, see the beginning of section \ref{sec:prelim}.

\begin{thmalpha}[Theorem \ref{thm:main}]
    Let $\alpha_P,\alpha_Q$ and $\alpha_R$ be a triple of involutions of $\PGL(2,q)$, whose respective centers are $P,Q$, and $R$ and let $H_0 = \langle \alpha_Q,\alpha_R \rangle$, $H_1 = \langle \alpha_P \alpha_R \rangle$, $H_2 = \langle \alpha_P, \alpha_R \rangle$ and $H = \langle \alpha_P,\alpha_Q,\alpha_R \rangle$. Then the followings are equivalent
    \begin{enumerate}
        \item The points $P,Q$ and $R$ are not on a line, and the triangle $\{\alpha_P,\alpha_Q,\alpha_R\}$ of involutions is strongly non-self polar,
        \item The coset geometry $\Gamma = \Gamma(H,(H_0,H_1,H_2))$ is an hypertope of rank $3$.
    \end{enumerate}
\end{thmalpha}

In particular, this theorem can be used to show the existence of hypertopes, with non linear diagrams, of rank $3$ for the groups $\PGL(2,q)$.

\begin{thmalpha}[Corollary \ref{coro:PGLnonLinear}]
    For any $q = p^n$ for some odd prime $p$, there exist an hypertope of rank $3$ whose diagram is not linear and whose automorphism group is $\PGL(2,q)$.
\end{thmalpha}

Finally, we look at the special case where the triple of points $\{P,Q,R\}$ forms a triangle made of tangent lines to the conic $\OO$. In this case, the group generated by the three involutions $\alpha_P,\alpha_Q$ and $\alpha_R$ must be $\PGL(2,p)$ or $\PSL(2,p)$, no matter the value of $q = p^n$ we started from. The associated hypertopes $\Gamma$ also necessarily have full correlation group. In the case where the starting projective plane $\pi$ admits a collineations $\tau$ of order $3$ which is not a projectivity (i.e: q is a cube), we remark that there must exist a subplane $\PGL(2,p) \subset \pi$ on which $\tau$ is in fact a projectivity.

\begin{thmalpha}[Theorem \ref{thm:generationTangent}, Corollaries \ref{coro:tangent} and \ref{coro:coliToproj}]
    Let $\alpha_P,\alpha_Q$ and $\alpha_R$ be a triple of involutions of $\PGL(2,q)$, whose respective centers are $P,Q$ and $R$, such that the triangle $\{P,Q,R\}$ is tangent to the conic $\OO$. The followings hold:
    \begin{enumerate}
        \item The group $H = \langle \alpha_P,\alpha_Q,\alpha_R \rangle$ is isomorphic to $\PGL(2,p)$ or to $\PSL(2,p)$, depending on whether $p \equiv 1 \mod 4$ or $p \equiv 3 \mod 4$,
        \item If $H = \PGL(2,q)$, the coset geometry $\Gamma = \Gamma(H,(H_0,H_1,H_2))$ is an hypertope of rank $3$ such that $\Cor(\Gamma)/ \Aut(\Gamma) = S_3$, where $H_0 = \langle \alpha_Q,\alpha_R \rangle, H_1 = \langle \alpha_P \alpha_R \rangle$ and $H_2 = \langle \alpha_P, \alpha_R \rangle$. 
        \item If $q = q_0^3$, the points $P,Q$ and $R$ can be chosen such that they are an orbit of a triality $\tau$. Moreover, the action of $\tau$ on the projective subplane $\PG(2,p) \subset \pi $ corresponding to $H$ is a projectivity.
    \end{enumerate}
\end{thmalpha}

\section{Background}
In this section, we recall some of the necessary background about projectivities and collineations of a projective plane, the subgroup structure of $\PGL(2,q)$, incidence geometries and $C$-groups.

\subsection{Collineations and projectivities of the projective plane}\label{subsec:collineations}

Let $\pi = \PG(2,q)$, be a projective plane over the finite field $\mathbf{K} =\FF_q$, where $q = p^n$ for some prime $p$ and some integer $n$. A \textit{collineation} of $\pi$ is a bijective function $\varphi \colon \pi \to \pi$ of the points of $\pi$ that also preserves lines. A collineation of $\pi$ is a \textit{projectivity} if its action on $\pi$ can be induced from a linear map on the vector space $\mathbf{K}^3$. After choosing a projective basis for $\pi$, the group of projectivities of $\pi$ can be identified with $\PGL(3,q)$ and the group of all collineations of $\pi$ is then $\PGamma(3,q) = \langle \PGL(3,q), \varphi \rangle$, where $\varphi$ is the collineation of $\pi$ induced by the Frobenius automorphism in the chosen basis.

Let $\OO$ be a non-degenerate conic in $\pi$. The group of projectivities of $\pi$ sending $\OO$ to itself is a subgroup $G < \PGL(3,q)$ which, by definition, is isomorphic to $\PGO(3,q)$. This group $G$ can in fact also be identified with $\PGL(2,q)$.

\begin{lemma}\cite[Corollary 7.14]{hirschfeld1998projective}
    The group $G$ of projectivities fixing $\OO$ is isomorphic to $\PGL(2,q)$.    
\end{lemma} \label{lem:stabConic} 
Throughout this article, the plane $\pi$ and the conic $\OO$ will be fixed, and we will thus always see $\PGL(2,q)$ as the stabilizer in $\pi$ of $\OO$, which will always be denoted as $G$. We now recall a description of the involutions of $\PGL(2,q)$ in this representation.

A projectivity $\gamma$ of $\pi$ is called a \textit{perspectivity} if it sends each line going through a given point $P$ to itself. This point $P$ is then called the \textit{center} of $\gamma$. The perspectivity $\gamma$ must then also fix (point-wise) a line $l$, which is called the \textit{axis} of $\gamma$.

It is easy to show that if $\alpha$ is a perspectivity of $\pi$ fixing $\OO$, then $\alpha$ is an involution, its center is not on $\OO$ and its axis is determined by its center. Moreover, if  $\alpha$ and $\beta$ are two perspectivities fixing $\OO$, both having center $P$, they must in fact be equal.
In particular, all perspectivities of $\PGL(2,q) \subset \PGL(3,q)$ are all involutions. The converse in also true.

\begin{proposition}\cite[Corollary 5]{baer1946projectivities}
    Let $\alpha$ be an involution of a projective plane $\pi$ of order $q$, then either $q$ is a square and $\alpha$ is a Baer involution, or $\alpha$ is a perspectivity.
\end{proposition}

Let $I = \{\gamma \in G \mid O(\gamma) = 2\}$ be the set of involutions of $G$. Since every involution of $G$ is a perspectivity, we can define a map $C \colon I \to \pi \setminus \OO$ by sending an involution $\alpha \in I$ to its center. This map is in fact injective, as no two involutions of $G$ can have the same center. A counting argument on the number of involutions of $G$ shows that it is also surjective. We can thus associate to any point $P \in \pi \setminus \OO$ a unique involution $\alpha_P \in G$ such that $C(\alpha_P) = P$. For any subgroup $H \subset G$, we will say that \textit{P belongs to $H$} if $\alpha_P \in H$. Note that there is a dual map $l \colon I \to \pi^* \setminus \OO^*$ where $\pi^*$ and $\OO^*$ are the duals of $\pi$ and $\OO$. This map $l$ associates to each involutions $\alpha$ its axis $l(\alpha)$.

A line $l$ is \textit{secant} (respectively \textit{tangent}, \textit{exterior}) to $\OO$ if it meets $\OO$ in $2$ (respectively, 1,0) points. 
A point $P \in \pi \setminus \OO$ is said to be \textit{exterior} to $\OO$ if there exists a tangent line $l$ to $\OO$ containing $P$. Else, $P$ is said to be \textit{interior} to $\OO$. 

Let $P$ be an exterior point. Then, the involution $\alpha_P$ has two fixed point on $\OO$. Indeed, $\alpha_P$ must fix the points on intersection of $\OO$ with the tangents to $\OO$ going through $P$. On the contrary, if $P$ in interior, $\alpha_P$ has no fixed points on $\OO$. The amount of fixed points in $\OO$ of an involution $\alpha$ determines whether $\alpha$ is in $\PSL(2,q) \subset \PGL(2,q)$ or not.

\begin{lemma}
    Let $H = \PSL(2,q)$. Then
    \begin{enumerate}
        \item If $q = 1 \pmod 4$, any involution of $H$ has $2$ fixed points on $\OO$,
        \item If $q = 3 \pmod 4$, any involution of $H$ has no fixed points on $\OO$.
    \end{enumerate}
\end{lemma}\label{lem:fixedpointsonO}

We can thus see that if $q = 1 \pmod 4$, all the exterior points belong to $\PSL(2,q)$ while, if $q = 3 \pmod 4$, it is the interior points that belong to $\PSL(2,q)$.

The projective plane $\pi$ also admits another interesting type of symmetries, which are called \textit{polarities}. Polarities send points of $\pi$ to lines of $\pi$ and vice-versa, while preserving incidence, and have order $2$. In the language of incidence geometry, polarities are correlations of $\pi$ of order $2$ that are not automorphisms.

Polarities are closely related to conics (see \cite{coxeter2003projective}, for example, for more details). Let $\psi$ be a polarity and let $P$ and $l$ be a point and a line of $\pi$, respectively. The line $\psi(P)$ is called the \textit{polar} of $P$ and the point $\psi(l)$ is called the \textit{pole} of $l$. A point $P$ of $\pi$ is said to be \textit{self-polar} if it lies on its polar. It can be shown that the set of self-polar points of a polarity is always a conic. Conversely, given a conic $\OO$, we can construct a polarity. If $P \in \OO$, we set that the tangent to $\OO$ at $P$ is the polar of $P$. For any other point $P$, take $l_1$ and $l_2$ to be two secants to $\OO$ containing $P$. Let $P_1$ and $Q_1$ be the intersections of $\OO$ with $l_1$ and let $P_2$ and $Q_2$ be the intersection of $\OO$ with $l_2$. The polar of $P$ in then the line joining $Q_1Q_2 \cap P_1P_2$ and $Q_1P_2 \cap P_1Q_2$.

There is an interplay between geometry and algebra in the following sense. Let $\psi$ is the polarity associated to our choice $\OO$ of conic. The axis of the involution $\alpha_P \in G$ is none other that the polar $\psi(P)$ of $P$.

Finally, we will say that a triangle $\Delta \in \pi$ is \textit{self-polar} if each side of $\Delta$ is the polar of its opposite vertex.

\subsection{Subgroup structure of $\PGL(2,q)$ and $\PSL(2,q)$}

The subgroup structure of $\PSL(2,q)$ can be found in the book by Dickson \cite{Dickson1958LinearGW} (but was first studied by Moore \cite{moore1903subgroups} and Wiman \cite{wiman1899bestimmung}). As $\PGL(2,q)$ can always be embedded into $\PSL(2,q^2)$, the subgroup structure of $\PGL(2,q)$ can be deduced from the one of $\PSL(2,q)$.

We will only need to know the list of the maximal subgroups of $\PSL(2,q)$ and of $\PGL(2,q)$. While such a list can be deduced from \cite{Dickson1958LinearGW}, more compact statements can be found in \cite{giudici2007maximal}. We start with the statement for $\PGL(2,q)$.

\begin{theorem}\label{thm:MaxSubPGL}
    Let $G = \PGL(2,q)$ with $q = p^n > 3$ for some odd prime $p$. Then the maximal subgroups of $G$ not containing $\PSL(2,q) are:$
    \begin{enumerate}
        \item $C_p^n : C_{q-1}$,
        \item $D_{2(q-1)}$ for $q \neq 5$
        \item $D_{2(q+1)}$,
        \item $S_4$ for $q = p \equiv \pm 3 (\bmod\; 8)$,
        \item $\PGL(2,q_0)$ for $q = q_0^r$ with $r$ and odd prime.
    \end{enumerate}
\end{theorem}

Remark that under the representation of $G = \PGL(2,q)$ that we chose, the first three type of maximal subgroups correspond to stabilizers of tangent, exterior and secant lines, respectively. Moreover, the stabilizers of tangent lines are isomorphic to $\AGL(1,q)$, the group of affinities of an affine line with $q$ elements.

\begin{theorem}\label{thm:MaxSubPSL}
Let $q = p^n > 5$ with $p$ and odd prime. Then the maximal subgroups of $\PSL(2,q)$ are:
\begin{enumerate}
    \item $C_p^n : C_{q-1/2}$, the stabilizer of a point of a projective line,
    \item $D_{q-1}$ for $q \geq 13$
    \item $D_{q+1}$, for $q \neq 7,9$,
    \item $\PGL(2,q_0)$ for $q = q_0^2$,
    \item $\PSL(2,q_0)$ for $q = q_0^r$ where $r$ is an odd prime,
    \item $A_5$ for $q \equiv \pm1 (\bmod\; 10)$ where either $q = p$ or $q = p^2$ and $p \cong \pm 3 (\bmod\; 10)$,
    \item $A_4$ for $q = p \equiv \pm 3 (\bmod\; 8)$ and $q \not\equiv \pm 1 (\bmod\; 10)$,
    \item $S_4$ for $q = p \equiv \pm 1 (\bmod\; 8)$.
\end{enumerate}
\end{theorem}

The subgroups of type $\PGL(2,q_0)$ and $\PSL(2,q_0)$ for some $q_0$ diving $q$ will be called \textit{subfield subgroups} of $G$.

\subsection{Incidence and coset geometries}
To their core, most of the geometric objects of interest to mathematicians are composed of elements together with some relation between them. This very general notion is made precise by the notion of an incidence system, or an incidence geometry. For a more detailed introduction to incidence geometry, we refer to~\cite{buekenhout2013diagram}.

    A triple $\Gamma = (X,*,t)$ is called an \textit{incidence system} over $I$ if
    \begin{enumerate}
        \item $X$ is a set whose elements are called the \textit{elements} of $\Gamma$,
        \item $*$ is a symmetric and reflexive relation (called the \textit{incidence relation}) on $X$, and
        \item $t$ is a map from $X$ to $I$, called the \textit{type map} of $\Gamma$, such that distinct elements $x,y \in X$ with $x * y$ satisfy $t(x) \neq t(y)$. 
    \end{enumerate}
Elements of $t^{-1}(i)$ are called the elements of type $i$.
The \textit{rank} of $\Gamma$ is the cardinality of the type set $I$.
A \textit{flag} in an incidence system $\Gamma$ over $I$ is a set of pairwise incident elements. The type of a flag $F$ is $t(F)$, that is the set of types of the elements of $F.$ A \textit{chamber} is a flag of type $I$. An incidence system $\Gamma$ is an \textit{incidence geometry} if all its maximal flags are chambers.

Let $F$ be a flag of $\Gamma$. An element $x\in X$ is {\em incident} to $F$ if $x*y$ for all $y\in F$. The \textit{residue} of $\Gamma$ with respect to $F$, denoted by $\Gamma_F$, is the incidence system formed by all the elements of $\Gamma$ incident to $F$ but not in $F$. The \textit{rank} of the residue $\Gamma_F$ is equal to rank$(\Gamma)$ - $|F|$.

The \textit{incidence graph} of $\Gamma$ is a graph with vertex set $X$ and where two elements $x$ and $y$ are connected by an edge if and only if $x * y$. Whenever we talk about the distance between two elements $x$ and $y$ of a geometry $\Gamma$, we mean the distance in the incidence graph of $\Gamma$ and simply denote it by $d_\Gamma(x,y)$, or even $d(x,y)$ if the context allows.

An incidence geometry $\Gamma$ is \textit{connected} if its incidence graph is connected. It is \textit{residually connected} if all its residue of rank at least two are connected. It is \textit{thin} if all its residue of rank two contain exactly two elements.

Let $\Gamma = \Gamma(X,*,t)$ be an incidence geometry over the type set $I$. A {\em correlation} of $\Gamma$ is a bijection $\phi$ of $X$ respecting the incidence relation $*$ and such that, for every $x,y \in X$, if $t(x) = t(y)$ then $t(\phi(x)) = t(\phi(y))$. If, moreover, $\phi$ fixes the types of every element (i.e $t(\phi(x)) = t(x)$ for all $x \in X$), then $\phi$ is said to be an {\em automorphism} of $\Gamma$. The \emph{type} of a correlation $\phi$ is the permutation it induces on the type set $I$. A correlation of type $(i,j)$ is called a duality if it has order $2$ and a correlation of type $(i,j,k)$ is called a triality if it has order $3$. The group of all correlations of $\Gamma$ is denoted by $\Cor(\Gamma)$ and the automorphism group of $\Gamma$ is denoted by $\Aut(\Gamma)$. Remark that $\Aut(\Gamma)$ is a normal subgroup of $\Cor(\Gamma)$ since it is the kernel of the action of $\Cor(\Gamma)$ on $I$.

If $\Aut(\Gamma)$ is transitive on the set of chambers of $\Gamma$ then we say that $\Gamma$ is {\em flag transitive}. If moreover, the stabilizer of a chamber in $\Aut(\Gamma)$ is reduced to the identity, we say that $\Gamma$ is {\em simply transitive} or {\em regular}.

Francis Buekenhout introduced in~\cite{buek} a new diagram associated to incidence geometries. His idea was to associate to each rank two residue a set of three integers giving information on its incidence graph.
Let $\Gamma$ be a rank $2$ geometry. We can consider $\Gamma$ to have type set $I = \{P,L\}$, standing for points and lines. The {\em point-diameter}, denoted by $d_P(\Gamma) = d_P$, is the largest integer $k$ such that there exists a point $p \in P$ and an element $x \in \Gamma$ such that $d(p,x) = k$. Similarly the {\em line-diameter}, denoted by $d_L(\Gamma) = d_L$, is the largest integer $k$ such that there exists a line $l \in L$ and an element $x \in \Gamma$ such that $d(l,x) = k$. Finally, the \textit{gonality} of $\Gamma$, denoted by $g(\Gamma) = g$ is half the length of the smallest circuit in the incidence graph of $\Gamma$.

If a rank $2$ geometry $\Gamma$ has $d_P = d_L = g = n$ for some natural number $n$, we say that it is a \textit{generalized $n$-gon}. Generalized $2$-gons are also called generalized digons. They are in some sense trivial geometries since all points are incident to all lines. Their incidence graphs are complete bipartite graphs. Generalized $3$-gons are projective planes.

Let $\Gamma$ be a geometry over $I$.  The \textit{Buekenhout diagram} (or diagram for short) $D$ for $\Gamma$ is a graph whose vertex set is $I$. Each edge $\{i,j\}$ is labeled with a collection $D_{ij}$ of rank $2$ geometries. We say that $\Gamma$ belongs to $D$ if every residue of rank $2$ of type $\{i,j\}$ of $\Gamma$ is one of those listed in $D_{ij}$ for every pair of $i \neq j \in I$. In most cases, we use conventions to turn a diagram $D$ into a labeled graph. The most common convention is to not draw an edge between two vertices $i$ and $j$ if all residues of type $\{i,j\}$ are generalized digons, and to label the edge $\{i,j\}$ by a natural integer $n$ if all residues of type $\{i,j\}$ are generalized $n$-gons. It is also common to omit the label when $n=3$.
If the edge $\{i,j\}$ is labeled by a triple $(d_{ij},g_{ij},d_{ji})$ it means that every residue of type $\{i,j\}$ had $d_P = d_{ij}, g = g_{ij}, d_L = d_{ji}$. We can also add information to the vertices of a diagram. 
We can label the vertex $i$ with the number $n_i$ of elements of type $i$ in $\Gamma$. Moreover, if for all flags $F$ of co-type $i$, we have that $|\Gamma_F| = s_i +1$, we will also label the vertex $i$ with the integer $s_i$.

The geometries of interest to us in this article are the so-called \textit{hypertopes}, introduced in \cite{HighlySymmetricHypertopes}. These are thin, residually connected and flag-transitive geometries. Contrary to abstract polytopes, hypertopes do not generally have linear diagrams.

Incidence geometries can be obtained from a group $G$ together with a set $(G_i)_{i \in I}$ of subgroups of $G$ as described in~\cite{Tits1957}. 
    The \emph{coset geometry} $\Gamma(G,(G_i)_{i \in I})$ is the incidence geometry over the type set $I$ where:
    \begin{enumerate}
        \item The elements of type $i \in I$ are right cosets of the form $G_i \cdot g$, $g \in G$.
        \item The incidence relation is given by non empty intersection. More precisely, the element $G_i \cdot g$ is incident to the element $G_j \cdot k$ if and only if $i\neq j$ and $G_i \cdot g \cap G_j \cdot k \neq \emptyset$.
    \end{enumerate}

We can check flag transitivity and residual connectedness of a coset geometry $\Gamma(G,(G_i)_{i \in I})$ by using group theoretical conditions. For any $J \subset I$, we define $G_J = \cap_{j \in J} G_j$.

\begin{theorem} \cite[Theorem 1.8.10]{buekenhout2013diagram}\label{thm:FT}
    Let $\Gamma = \Gamma(G,(G_i)_{i \in I})$ be a coset incidence geometry. Then $G$ acts flag-transitively on $\Gamma$ if and only if for each $J \subset I$ and for each $i \in I\setminus J$, we have that $G_JG_i= \cap_{j\in J} (G_jG_i)$.
\end{theorem}

For geometries of rank $3$, Theorem \ref{thm:FT}, reduces to check, for each $\{i,j,k\} = I$, that $(G_i \cap G_j) (G_i \cap G_k) =G_i \cap (G_jG_k)$. This was already known to Tits~\cite[Section 1.4]{Tits1974}. Once we know that $G$ acts flag-transitively on $\Gamma$, we can check residual connectedness using the following Theorem.

\begin{theorem} \cite[Corollary 1.8.13]{buekenhout2013diagram}\label{prop:RC}
    Suppose $I$ is finite and let $\Gamma = \Gamma(G,(G_i)_{i \in I})$ be a geometry on which $G$ acts flag-transitively. Then $\Gamma$ is residually connected if and only if $G_J = \langle G_{J \cup i} \mid i \in I\setminus J \rangle$ for each $J \subset I$ such that $|I\setminus J| \geq 2$.
\end{theorem}

\subsection{C-groups}

A \textit{C-group of rank $r$} is a pair $(G,S)$ where $G$ is a group and $S = \{\alpha_0,\alpha_1,\cdots, \alpha_{r-1} \}$ is a set of $r$ involutions such that $G= \langle \alpha_0,\alpha_1,\cdots, \alpha_{r-1}\rangle$ and such that for all subsets $J,K \subset I = \{0,1\cdots,r-1\}$ we have

\begin{equation*}
    \langle \alpha_j \mid j \in J \rangle \cap \langle \alpha_k \mid k \in K \rangle = \langle \alpha_l \mid l \in J \cap K \rangle.
\end{equation*}
This second property is called the \textit{intersection property}. We can immediately remark that it is very similar to the criterion for residual connectedness of Theorem \ref{prop:RC}. 

In fact, automorphism groups of regular hypertopes are always $C-$groups. The converse is not true (see \cite[Example 4.4]{HighlySymmetricHypertopes}, for example). 

Here is how $C-$groups will come into play in this article.
Let $S =\{\alpha_0,\alpha_1,\cdots, \alpha_{r-1} \}$ be a set of involutions generating a group $G$ and let, for each $i\in I$, $G_i$ be the subgroup of $G$ generated by all the involutions of $S$ except from $\alpha_i$. We can then consider the coset geometry $\Gamma = \Gamma (G,(G_i)_{i \in I})$. Then $\Gamma$ is a regular hypertope if and only if $(G,S)$ is a $C$-group and the condition of Theorem \ref{thm:FT} holds.
The diagram $D$ of $\Gamma = \Gamma (G,(G_i)_{i \in I})$ can be computed easily. Each vertex of $D$ correspond to one of the involutions of $S$ and the edge $\{i,j\}$ is labeled with the order of $\alpha_i\alpha_j$. If the order is $2$, the edge is omitting entirely.

\section{Preliminaries}\label{sec:prelim}

In this section, we set the notations that we will use in the rest of the parper, and we show some preliminary results.

Let $\pi = \PG(2,q)$ be a projective plane of order $q$ where $q = p^n$ for some odd prime $p$ and some integer $n$. Let $\OO$ be a non-degenerate conic in $\pi$. The conic $\OO$ can be considered as the set of self-polar points of a polarity $\psi$. Let $G$ be the stabilizer in $\PGL(3,q)$ of $\OO$. Then $G$ is isomorphic to $\PGL(2,q)$, as shown in Proposition \ref{lem:stabConic}. Let $\alpha \in G$ be an involution, which is then a perspectivity with a center and an axis that we denote by $C(\alpha)$ and $l(\alpha)$, respectively. Remark that $l(\alpha)$ is the polar of $C(\alpha)$ and that $C(\alpha)$ is the pole of $l(\alpha)$. For any $P \in \pi \setminus \OO$, the unique involution of $\PGL(2,q)$ with center $P$ will be denoted by $\alpha_P$ (see section \ref{subsec:collineations} for more explanations).

In the whole article, points will be denoted by upper case letters ($P,Q,R$ ...), lines by lower case letter ($l,m,k$, ...), and the line connecting two points $P$ and $Q$ will be denoted by $(PQ)$. Also, for an element $g \in G$, we will denote the order of $g$ by $O(g)$.

We want to construct thin, residually connected and flag-transitive geometries (i.e: regular hypertopes) of rank $3$ from triples of involutions of $G$. More precisely, we will choose three involutions $\alpha_P,\alpha_Q$ and $\alpha_R$ and construct the coset geometry $\Gamma(H,(H_0,H_1,H_2))$ where $H = \langle \alpha_P,\alpha_Q, \alpha_R \rangle$ and $H_0 = \langle \alpha_Q,\alpha_R \rangle, H_1 = \langle \alpha_P,\alpha_R \rangle, H_2 = \langle \alpha_P,\alpha_Q \rangle$. We will give geometric conditions on the configuration of the three centers $P,Q,R$ that will guarantee that $\Gamma$ is indeed an hypertope.

Let $\Delta =\{\alpha_P,\alpha_Q,\alpha_R \}$ be a triple of involutions in $G$. We say that $\Delta$ is a \textit{triangle of involutions} (or just a triangle) if the centers $P.Q$ and $R$ of the three involutions do not lie on a line of $\pi$.
The triangle $\Delta$ is \textit{proper} if the polar of each of the three points $P,Q$ and $R$ is never the line between the other two. 
We say that $\Delta$ is \textit{self-polar} if the axis of each involution of $\Delta$ is the line joining the center of the two other involutions. Therefore, a self-polar triangle is never proper, but a triangle which is not proper does not have to be self-polar.
The \textit{sides} of $\Delta$ are the three sets $(\alpha_P,\alpha_Q) = \{C(g) \mid g \in \langle \alpha_P,\alpha_Q \rangle \},(\alpha_Q,\alpha_R) = \{C(g) \mid g \in \langle \alpha_Q,\alpha_R \rangle \}$ and $(\alpha_P,\alpha_R) = \{C(g) \mid g \in \langle \alpha_P,\alpha_R \rangle \}$. We say that a point $X$ is in $\Delta$, and denote it simply by $X \in \Delta$, if $X$ is a point of one of the sides of $\Delta$. Finally, we say that $\Delta$ is \textit{strongly non self-polar} if there is no self-polar triangle $\{X,Y,Z\}$ such that its vertices belong to different sides of $\Delta$. In particular, a strongly non self-polar triangle cannot be self-polar. Note that $\Delta =\{\alpha_P,\alpha_Q,\alpha_R \}$ is self-polar if and only if $H = \langle \alpha_P,\alpha_Q,\alpha_R \rangle$ is of order $4$ and the three involutions all commute with each other.

If we choose three involutions $\alpha_P,\alpha_Q$ and $\alpha_R$ such that $P,Q,R$ are on a line $l$, the group $H = \langle \alpha_P,\alpha_Q,\alpha_R \rangle$ must be contained in the stabilizer of $l$. The same is true for triangles of involutions that are not proper.

\begin{lemma}\label{lem:insidestab}
    Let $\Delta = \{ \alpha_P,\alpha_Q , \alpha_R\}$ is a triangle which is not proper. Then $H = \langle \alpha_P,\alpha_Q , \alpha_R \rangle $ is inside the stabilizer of one of the lines $(PQ),(PR)$ or $(QR)$. Moreover, $H$ is isomorphic to $C_2 \times D_n$ or to $D_n$, for some integer $n$.
\end{lemma}

\begin{proof}
    Suppose that the axis of $\alpha_P$ is the line $(QR)$. Then, clearly, all three of $\alpha_P,\alpha_Q$ and $\alpha_R$ stabilize $(QR)$. Moreover, $\alpha_P$ commutes with $\alpha_Q$ and $\alpha_R$ since $\alpha_P$ is the central involution in the stabilizer of $(QR)$. Hence, $H = \langle \alpha_P \rangle \times \langle \alpha_Q,\alpha_R \rangle$ if $\alpha_P$ is not in $\langle \alpha_Q,\alpha_R \rangle$ and $H = \langle \alpha_Q,\alpha_R \rangle$ otherwise.
\end{proof}

We now establish precise results on the intersections of stabilizers of lines.

\begin{proposition}\label{prop:stabilizerIntersection} 
     
     Let $l_1$ and $l_2$ be two distinct lines of $\pi$ and let $G_1$ and $G_2$ be their respective stabilizers in $G$. Let $P$ be the intersection of $l_1$ and $l_2$.
     \begin{enumerate}
         \item If $l_1$ and $l_2$ are both not tangent to $\OO$, then either $G_1 \cap G_2 = \langle \alpha_P \rangle$ or $G_1 \cap G_2$ is a dihedral group of order $4$ containing $\alpha_P$. In the latter case, $G_1 \cap G_2 = \{e,\alpha_P,\alpha_Q,\alpha_R\}$ where $Q = l_1 \cap l(\alpha_P)$ and $R = l_2 \cap l(\alpha_P)$.
         \item If $l_1$ is tangent and $l_2$ is not, then $G_1 \cap G_2$ is either a cyclic group of order $2$ if $P \notin \OO$ or a cyclic group of order $q-1$ if $P \in \OO$.
         \item If both $l_1$ and $l_2$ are tangents, then $G_1 \cap G_2$ is a cyclic group of order $q-1$.
     \end{enumerate}

\end{proposition}

\begin{proof}
Let $H = G_1 \cap G_2$. 

(1) We investigate the involutions in $H$. The involutions in $G_1$ correspond to the perspectivities of $G$ with center on $l_1$, together with the unique perspectivity with axis $l_1$. Clearly, the perspectivity $\alpha_P$ with center $P = l_1 \cap l_2$ also fixes $l_2$ and thus is in $H$. Suppose there is another perspectivity $\alpha_Q \in H$. Then either $Q \in l_2$, and then the axis $l(\alpha_Q)$ must be $l_1$ (since $\alpha_Q \in G_1$), or the axis of $\alpha_Q$ is $l_2$ and $Q \in l_1$. Since there is a unique involution in $G$ having a given line $l$ as an axis, we have shown that $H$ contains at most three involutions and at least one.

Let $h \in G_1$ be a generator for the maximal cyclic subgroup of $G_1$ (recall that the $G_i$ are dihedral groups). Then $h$ can be written as $\alpha_P \alpha_X$ for at least two choices of point $X \in l_1$.
Therefore, we can always chose $\alpha_X$ not in $H$, since at most two involutions of $H$ have center on $l_1$. This shows that $h$ cannot be in $H$ and thus that $H$ contains only involutions. 

In conclusion, we see that $H$ is indeed either $\langle \alpha_P \rangle$ or an elementary abelian group of order $4$ containing three involutions $\alpha_P, \alpha_Q$ and $\alpha_R$ that pairwise all commute. Therefore, the axis of any of these involutions must contain the center of the two others.

(2) Suppose first that $P$ is not on $\OO$. Then the involution $\alpha_P$ is in $H$. There can be no other involution in $H$. Indeed, if the involution with axis $l_2$ were in $H$, the same argument as in (1) would mean that there must be a third involution in $H$. The only candidate is the involution with axis $l_1$, but since $l_1$ is tangent, that is not possible. By point (1), the only elements of $G_2$ that can fix $P$ are involutions, so clearly there can be only involutions in $H$. 

If instead $P \in \OO$, then the only involution in $H$ is the one with axis $l_2$. Indeed, its center must be on $l_1$ since $l_1$ is the polar of $P$. All other involutions of $G_2$ exchange $P$ with the other point of intersection of $l_2$ with $\OO$. Hence, any product of two such involutions fixes $P$, and is thus in $G_1$. The group $H$ is thus equal to the unique cyclic subgroup of order $q-1$ of $G_2$.

(3) In this case, both $G_1$ and $G_2$ are isomorphic to $\AGL(1,q)$, the group of affinities of an affine line with $q$ elements. Their intersection is then isomorphic to the group of linear maps of $\FF_q$, which is cyclic of order $q-1$, as it is the multiplicative group $\FF^*_q$.

\end{proof}

\section{Hypertopes of rank $3$ from triples of involutions in $\PGL(2,q)$}

In this section, we classify all hypertopes of rank $3$ whose automorphism group is generated by three involutions of $\PGL(2,q)$. We start by two lemmas, that take care of the cases in which the $3$ involutions are all contained in the stabilizer of some line $l$. The first one handles the case where $l$ is exterior or secant to $\OO$, meaning that its stabilizer is a dihedral group.

\begin{lemma}\label{lem:HypertopeDihedral}
    Let $H$ be a dihedral group of order $n$. Then, there exists a unique, up to isomorphism, regular hypertope of rank $3$ with automorphism group $H$.
\end{lemma}

\begin{proof}
    Let $\alpha_0,\alpha_1$ and $\alpha_2$ be three involutions that generate $H$. In \cite{CgroupSuzuki}, it is shown that $(H,(\alpha_0,\alpha_1,\alpha_2))$ is a $C$-group only if one of the involutions, say $\alpha_0$, is central and $\alpha_0$ is not in $\langle \alpha_1,\alpha_2 \rangle$. Up to isomorphism, this is thus the only possible choice of generators.
    
    We define a geometry $\Gamma'$ as follows. Let $P$ be a polygon of with $n/2$ vertices and $n/2$ edges and let $Q$ and $R$ be two points. The elements of type $0$ of $\Gamma'$ are $Q$ and $R$, the elements of type $1$ are the vertices of $P$ and the elements of type $2$ are the edges of $P$. We declare that $Q$ and $R$ are incident to every elements of type $1$ and $2$ and that incidence between elements of type $1$ and $2$ is induced from the incidence in $P$.
    
    It can then be checked that $\Gamma = \Gamma(H,(H_0,H_1,H_2))$, where $H_i = \langle \alpha_j,\alpha_k \rangle$ with $\{i,j,k\} = \{0,1,2\}$, is isomorphic to $\Gamma'$. Since $\Gamma'$ is a clearly an hypertope, $\Gamma$ must also be an hypertope.
\end{proof}

We now look at the case where $l$ is tangent, and so its stabilizer is $\AGL(1,q)$.

\begin{lemma}\label{lem:HypertopeAGL}
    For any $q = p^n$ with $p$ an odd prime, there is no regular hypertope with automorphism group $H$ where $H$ is a subgroup of $\AGL(1,q)$ generated by involutions.
\end{lemma}

\begin{proof}
    Let $\mathbb{F}_q$ be the field of cardinality $q$, seen as a vector space of dimension $n$ over $\mathbb{F}_p$. Then, any $g \in \AGL(1,q)$ can be seen as a linear map on $\mathbb{F}_q$. In that perspective, we have that $g(x_1,x_2,\cdots,x_n) = a(x_1,x_2, \cdots,x_n) +b$ where $a,b \in \mathbb{F}_p$. Since the characteristic of $\mathbb{F}_q$ is odd, the involutions of $\AGL(1,q)$ are the elements with $a = -1$. Any involution is thus of the form $(x_1,x_2,\cdots,x_n) \to (-x_1,-x_2, \cdots,-x_n) +b$, for some $b \in \mathbb{F}_p$. We say that $b$ is the point associated to the involution. Let $S$ be a set of involutions of $\AGL(1,q)$. We say that this set $S$ is affinely independent if the set of associated points is affinely independent. It is proven in \cite{CgroupPGL} that $(H,S)$ is a $C-$group if and only if $S$ is affinely independent. 
    Suppose thus that $S = \{\alpha_0,\alpha_1,\cdots,\alpha_k\}$ is affinely independent. Let $\Gamma = \Gamma(H,(H_0,H_1,\cdots, H_{k}))$ where $H_i$ is the group generated by all involutions of $S$ except from $\alpha_i$. Elements of type $i$ in $\Gamma$ can be identified with families of parallel hyperplanes in the affine plane of dimension $k-1$ over $\mathbb{F_p}$. Hence, every element of type $i$ is incident to every elements of type $j$ for every $i \neq j$. Therefore, $\Gamma$ is clearly not thin, as every residue of rank $1$ contains $p > 2$ elements.
\end{proof}
We are now ready to prove our first main theorem, which classifies all triples of involutions that give rise to rank $3$ regular hypertopes.

\begin{figure}
    \centering
    \scalebox{0.85}{%
    \begin{tikzpicture}
    \node[anchor=south west,inner sep=0] (image) at (0,0) {\includegraphics[width=0.9\textwidth]{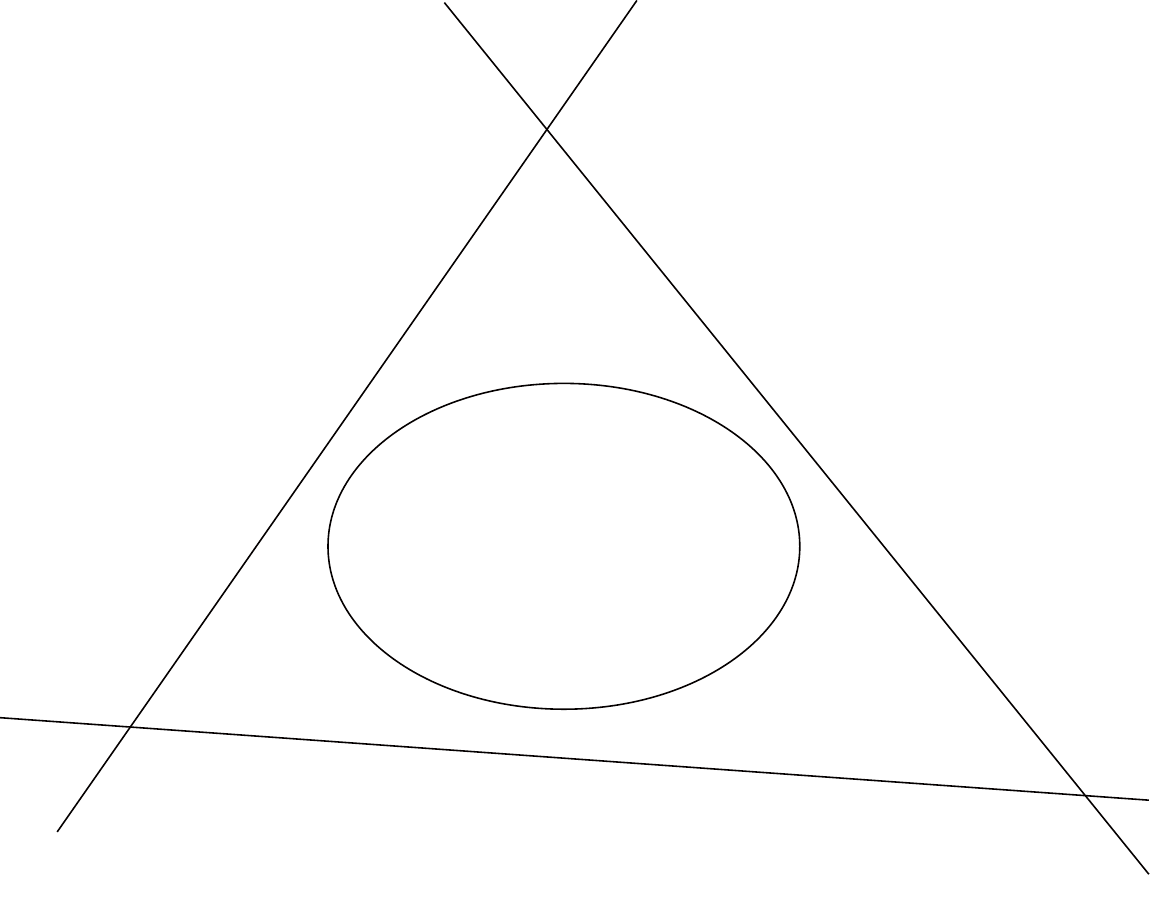}};
    \filldraw[black] (1.63,2.35) circle (2pt)  node[anchor=south]{$P$};
    \filldraw[black] (6.78,9.75) circle (2pt)  node[anchor=south]{$Q$};
    \filldraw[black] (13.42,1.54) circle (2pt)  node[anchor=south]{$R$};
    
    \filldraw[black] (4,6.3) circle (0pt)  node[anchor=east]{$l_2$};
    \filldraw[black] (11,6) circle (0pt)  node[anchor=east]{$l_0$};
    \filldraw[black] (8,1.5) circle (0pt)  node[anchor=east]{$l_1$};
    
    \end{tikzpicture}
    }
    \caption{ The triangle of points and lines in $\pi$ associate to $\Delta = \{\alpha_P,\alpha_Q,\alpha_R\}$.}
    \label{fig:triangleGeneral}

\end{figure}
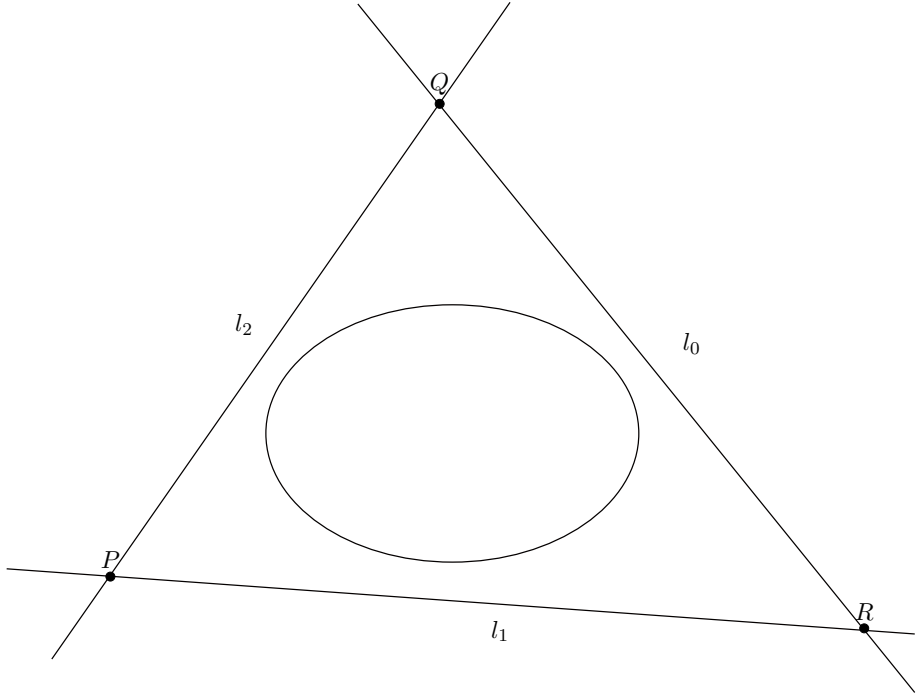

\begin{theorem}\label{thm:main}
     Let $S = \{P,Q,R\}$ be a triple of points in $\pi \setminus \OO$ and let $\Delta = \{ \alpha_P,\alpha_Q , \alpha_R\}$. Let $\Gamma = \Gamma (H,(H_0,H_1,H_2))$ be the coset geometry where $H_0 = \langle \alpha_Q,\alpha_R \rangle$, $H_1 = \langle \alpha_P,\alpha_R \rangle, H_2 = \langle \alpha_P,\alpha_Q \rangle$ and $H$ is $\langle \alpha_P,\alpha_Q,\alpha_R \rangle$.
     Then $\Gamma$ is a regular hypertope if and only if $\Delta$ is a strongly non self-polar triangle. 
\end{theorem}

\begin{proof}

    We divide the proof in three distinct cases, according to the configurations of the centers and axis of $\alpha_P, \alpha_Q$ and $\alpha_R$.
    
    First, suppose that $P,Q$ and $R$ are all on a line $l$. Then $\Delta$ is not a triangle, and we need to show that $\Gamma$ is not a regular hypertope.
    If the line $l$ is not tangent, $H$ is a subgroup of a dihedral group. Note that none of $\alpha_P,\alpha_Q$ and $\alpha_R$ is central in $H$, as it is impossible for the axis of one of the involutions to be the line $l$. Hence, by Lemma \ref{lem:HypertopeDihedral}, $\Gamma$ is not an hypertope.
    If instead $l$ is tangent, $H$ is then a subgroup of $\AGL(1,q)$ generated by involutions. Lemma \ref{lem:HypertopeAGL} shows that $\Gamma$ cannot be an hypertope.

    Suppose now that $\Delta$ is a triangle but that it is not proper. Let us further suppose, without loss of generality, that the axis of $\alpha_P$ is the line $(QR)$. Note that this implies that $(QR)$ is not tangent, as tangent lines are self-polar. Hence, by Lemma \ref{lem:HypertopeDihedral}, $\Gamma$ is an hypertope if and only if $\alpha_P$ is not in $\langle \alpha_Q,\alpha_R \rangle$. Now remark that the non proper triangle $\Delta$ is strongly non self-polar if and only if $\alpha_P$ is not in $\langle \alpha_Q,\alpha_R \rangle$. Therefore, in this case, we indeed have that $\Gamma$ is an hypertope if and only if $\Delta$ is strongly non self-polar.

    Finally, it remains to handle the case where $\Delta$ is a proper triangle. We will first show that if $\Delta$ is strongly non self-polar, then $\Gamma$ is a regular hypertope.
    
    Let $I = \{0,1,2\}$. For the sake of convenience, we relabel $\alpha_P = \alpha_0, \alpha_Q = \alpha_1$ and $\alpha_R = \alpha_2$ and we set $l_0 = (QR),l_1 = (PR)$ and $l_2 = (PQ)$ to be the three lines of $\Delta$. This configuration is represented in Figure \ref{fig:triangleGeneral}. It is recommended to look at this figure and visualize where the centers and axis of the involutions used in this proof are while reading the following paragraphs.
    
    We first show that $H_i \cap H_j = \langle \alpha_k \rangle$ for any $\{i,j,k\} = I$. Clearly, $\langle \alpha_k \rangle \subset H_i \cap H_j $. Hence, we only need to show the reverse inclusion.
    We will use Proposition \ref{prop:stabilizerIntersection}. If the lines $l_i$ and $l_j$ are both not tangent to $\OO$, Proposition \ref{prop:stabilizerIntersection} tells us that their stabilizers intersect in a cyclic group of order $2$, unless $\Delta$ is self-polar.
    Suppose now that $l_i$ is a tangent line. Then the intersection of the stabilizer of $l_i$ and the stabilizer of $l_j$ is always a cyclic group of order $2$ or a cyclic group $C_{q-1}$ of order $q-1$. But $H_i$ contains only involutions and elements of order $p$ so the intersection of $H_i$ and $H_j$ is at most the cyclic subgroup of order $2$ of $C_{q-1}$. Hence, by Proposition \ref{prop:RC}, we already see that $\Gamma$ will be residually connected, as long as $H$ acts flag-transitively on $\Gamma$.

    We now turn our attention to flag-transitivity. We will use Proposition \ref{thm:FT}, or more precisely, the rank $3$ version stated just after the proposition. For clarity, and without loss of generality, we will fix an order of the indices $i,j,k$ and show that

    \begin{equation} \label{eq:FT}
    H_1 \cap H_2 H_0 = (H_1 \cap H_2)(H_1 \cap H_0) = \{e,\alpha_0,\alpha_2,\alpha_0\alpha_2\} \subset H_1.
    \end{equation}
    
    The inclusion $\{e,\alpha_0,\alpha_2,\alpha_0\alpha_2\} \subset H_1 \cap H_2 H_0$ always holds, so we will only prove the reverse inclusion.
    Let $g  = g_2g_0$ with $g_2 \in H_2$ and $g_0 \in H_0$. We need to show that if $g_2g_0 \in H_1$, then it must be either trivial or equal to one of $\alpha_0, \alpha_2$ or $\alpha_0\alpha_2$. We will consider various cases, depending on the orders of $g_2$ and $g_0$. Note that we can always suppose that neither $g_2$ nor $g_0$ is the trivial element. Indeed, if instead, for example, $g_2$ were trivial, we would immediately get that $g \in H_1 \cap H_2$, so that $g$ is either trivial or equal to $\alpha_0$. 

    Suppose first that both $g_2$ and $g_0$ are involutions and let $X$ be the center of $g_2$, $Y$ be the center of $g_0$ and $l =(XY)$. We claim that, as long as $l$ is not equal to $l_1$, if $g =g_2g_0$ is in $H_1$, then $g$ must be an involution. Indeed, if $l$ is not equal to $l_1$, then $g \in H_1 \cap \langle g_2,g_0 \rangle$. By Proposition \ref{prop:stabilizerIntersection}, we know that the intersection of two dihedral subgroups of stabilizers of different lines is at most $C_2$, a cyclic group of order $2$. 
    So we know that, if $g_2$ and $g_0$ are involutions, $g = g_2g_0$ is also an involution.
    We now consider three cases, according to whether the involutions $g_2$ and $g_0$ are central or not in $H_2$ and $H_0$, and draw conclusions on $g = g_2g_0$ in each case.
    \begin{enumerate}
        \item Both $g_2$ and $g_0$ are not central. Then $X$ is on the side $(\alpha_0,\alpha_1)$ of $\Delta$ and $Y$ is on the side $(\alpha_1,\alpha_2)$ of $\Delta$. If $l =(XY)$ is equal to $l_1$, this must mean that $g_2 = \alpha_0$ and $g_0 = \alpha_2$ so that $g = \alpha_0\alpha_2$. If $l$ is not equal to $l_1$, we showed that $g$ must then be an involution of $H_1$. Note that $g$ must be central in $\langle g_2,g_0 \rangle$ as $g$ is an involution which is the product of two involutions. Therefore, for $g$ to be in $H_1$, we must have that the center $Z$ of $g$ is on $H_1$. But then, $\{X,Y,Z\}$ is a self-polar triangle, contradicting the strongly non self-polarity of $\Delta$.
        \item Both $g_2$ and $g_0$ are central. This means that the axis $l(g_2) = l_2$ and the axis $l(g_0) = l_0$. In this case $l = (XY)$ cannot be equal to $l_1$ as the triangle $\Delta$ is proper. Since $g$ must fix $Q$ and the line $l_1$, the only remaining possible case is that the center of $g$ is on $l_1$ and its axis contains $Q$. Hence, $g$ is an involution, all three of $g_2,g_0$ and $g$ belong to the polar of $Q$ and they all have centers on this polar. But three involution that all commute cannot all have center on the same line, as one of them must be central in the stabilizer of that line.
        \item Only $g_2$ is central and $g_0$ is not. Suppose first that $l =(XY)$ is equal to $l_1$. This means that the center $X$ of $g_2$ is on $l_1$ and that $g_0 = \alpha_2$. This implies that $g_2 \in H_1$. Let $l'$ be the polar of $Q$ and let $Z$ be the intersection of $l'$ with $l_2$. Then $\{X,Z,Q \}$ is a self-polar triangle contradiction the strongly non self-polarity of $\Delta$. Indeed, $X$ is on $(\alpha_0,\alpha_2)$ since $g_2 = \alpha_X$ is in $H_1$ and $Z$ must be on $(\alpha_0,\alpha_1)$ as $\alpha_Z$ is the product of $\alpha_1$ and $g_2$, which are both in $H_2$. If instead $l$ is not equal to $l_1$, then $g$ is an involution and we can go back to case 1) or 2) depending on whether $g_2g_0$ is central in $H_1$ or not.
    \end{enumerate}
    Putting together all the cases above, we showed that if both $g_2$ and $g_0$ are involutions, then $g_2g_0$ is either not in $H_1$ or is exactly one of $\{e,\alpha_0,\alpha_2,\alpha_0\alpha_2\}$, as we desired. More precisely, we showed that, if $g_2$ and $g_0$ are involutions, the only way that $g_2g_0 \in H_1$ is if $g_2 \in \{e,\alpha_0 \}$ and $g_0 \in \{e,\alpha_2\}$. We also remark, that, up to appropriately relabeling the indices, the same conclusion holds for the product of two involutions in $H_i$ and $H_j$ for any values of $i,j \in I$. 

Suppose now that neither $g_2$ nor $g_0$ are involutions. Therefore, we must have that $g_2 = (\alpha_0\alpha_1)^k$ and $g_0 = (\alpha_1\alpha_2)^m$
    for some integers $k = 1,2,\cdots, O(\alpha_0\alpha_1)$ and $m = 1,2, \cdots, O(\alpha_1\alpha_2)$. If $k=m=1$, we have that $g_2g_0 = \alpha_0\alpha_2$. Else, we obtain that

    $$
    g_2g_0 = (\alpha_0\alpha_1)^{k-1}(\alpha_0\alpha_2)(\alpha_1\alpha_2)^{m-1}.
    $$

    Note that $(\alpha_0\alpha_1)^{k-1}\alpha_0 = \alpha_0^h$ with $h \in H_2$ if $k$ is odd and $(\alpha_0\alpha_1)^{k-1}\alpha_0 = \alpha_1^h$ with $h$ in $H_2$ if $k$ is even. The same holds for $\alpha_2(\alpha_1\alpha_2)^{m-1}$. Therefore, we get that $g_2g_0$ can be rewritten as a product of two involutions, one in $H_2$ and one in $H_0$. Remark also that, since neither $g_2$ nor $g_0$ are trivial, the two involutions are also both not trivial. By the previous case, we can then conclude that, if $g \in H_1$, then $g=\alpha_0\alpha_2$. Once again, the same conclusion holds, up to relabeling of the indices, for the product of any two non involutions in $H_i$ and $H_j$ for $i,j \in I$.

    Lastly, suppose that $g_2$ is not an involution but that $g_0$ is. Suppose that $g_2g_0$ is also an involution. Then $g_0$ in an involution in $H_0$, $g_2g_0$ is an involution in $H_1$ and their product is equal to $g_2$. We can then conclude that $ g = g_2g_0 \in \{e,\alpha_0\}$. Suppose instead that $g_2g_0$ is not an involutions. Then $g_0 = g_2^{-1}(g_2g_0)$ with $g_2^{-1} \in H_2$ not an involution, $g_2g_0 \in H_1$ not an involution and $g_0 \in H_0$. By the previous case, we would conclude that $g_0 = \alpha_0\alpha_2$, which is not an involution, and we thus arrive at a contradiction.

    This shows that $\Gamma$ is flag-transitive, and thus also that $\Gamma$ is thin, since a rank $3$ flag-transitive coset geometry is thin if $H_i \cap H_j = C_2$. As mentioned in the beginning, Corollary \ref{prop:RC} then implies that $\Gamma$ is also residually connected. Hence, we showed that $\Gamma$ is a regular hypertope whenever $\Delta$ is a strongly non self-polar triangle.

    Conversely, suppose that $\Delta$ is proper but not strongly non-self polar. Let $A,B,C$ be a self-polar triangles such that $A,B$ and $C$ are on different sides of $\Delta$. Note that since $\Delta$ is proper, we can suppose that, say, $C$ is different from $P,Q$ and $R$. Then, we have, up to potential reordering, that $\alpha_A \alpha_B = \alpha_C$ with $\alpha_A \in H_2, \alpha_B \in H_0$ and $\alpha_C \in H_1$. This directly contradicts equation \ref{eq:FT} and thus shows that $\Gamma$ is not an hypertope.
    
\end{proof}

Checking that a triangle $\Delta$ of involutions is strongly non self-polar seems hard in general. That being said, it can be easily checked in some cases. 
\begin{proposition}\label{prop:notPSL}
    Let $S = \{P,Q,R\}$ be a triple of points in $\pi \setminus \OO$ such that $\Delta = \{ \alpha_P,\alpha_Q , \alpha_R\}$ is a triangle and such that none of $\alpha_P,\alpha_Q$ and $\alpha_R$ are in $\PSL(2,q)$. Then $\Delta$ is strongly non self-polar. 
\end{proposition}

\begin{proof}
    The group $\PSL(2,q)$ is a normal subgroup of index $2$ in $\PGL(2,q)$. For any $g \in \PGL(2,q)$, define a determinant function by setting $\det(g) = 1$ if $g \in \PSL(2,q)$ and $\det(g) = -1$ otherwise. This function is clearly multiplicative as it is just the group operation on the quotient $\PGL(2,q)/\PSL(2,q)$. Notice that, in each $H_i$, since $\det(\alpha_j) = -1 =\det(\alpha_k)$, all involutions except the central one will have determinant $-1$, and all other elements will have determinant equal to $1$. Suppose that you can find a self-polar triangle $\{X,Y,Z\}$ such that $\alpha_X,\alpha_Y$ and $\alpha_Z$ belong to $H_0,H_1$ and $H_2$ respectively. Then $\alpha_X \alpha_Y = \alpha_Z$ but $\det(\alpha_X) \det(\alpha_Y) = (-1) (-1) = 1 \neq \det(\alpha_Z)$, which is a contradiction.
\end{proof}

In particular, if $q = 3 \pmod 4$, any triple $\Delta = \{\alpha_P,\alpha_Q,\alpha_R \}$ with $P,Q,R$ the vertices of a triangle of tangent lines will satisfy all conditions of Theorem \ref{thm:main}. Such a $\Delta$ will be called a \textit{triangle tangent to $\OO$}.

\begin{theorem}
       Let $S = \{P,Q,R\}$ be a triple of points in $\pi \setminus \OO$ such that $\Delta = \{ \alpha_P,\alpha_Q , \alpha_R\}$ is a triangle tangent to $\OO$.  Let $H_0 = \langle \alpha_Q,\alpha_R \rangle$, $H_1 = \langle \alpha_P,\alpha_R \rangle$ and $H_2 = \langle \alpha_P,\alpha_Q \rangle$ and $H$ is $\langle \alpha_P,\alpha_Q,\alpha_R \rangle$.
       Then, if $q = 3 \pmod 4$, $\Gamma = \Gamma (H,(H_0,H_1,H_2))$ is a regular hypertope. 
\end{theorem}
\begin{proof}
     The points $P,Q$ and $R$ must be exterior points as they lie on tangent lines. Since $q = 3 \pmod 4$, we get, by Lemma \ref{lem:fixedpointsonO}, that $\alpha_P,\alpha_R$ and $\alpha_Q$ are not in $\PSL(2,q)$. Hence $\Delta$ is strongly not self-polar by Proposition \ref{prop:notPSL}. Moreover, $\Delta$ must be a triangle since it is not possible for $P,Q$ and $R$ to be on a line. Hence, $\Gamma$ is a regular hypertope by Theorem \ref{thm:main}.

\end{proof}

\section{Groups generated by non self-polar triangles of involutions}

For any proper and strongly non self-polar triangle of involutions  $\Delta = \{ \alpha_P,\alpha_Q , \alpha_R\}$, we constructed hypertopes as coset geometries with automorphism group the group $H$ generated by these $3$ involutions. We now investigate what that group $H$ can be. For the sake of generality, we work with non self-polar triangles.

\begin{theorem} \label{thm:generationGeneral}
 Let $S = \{P,Q,R\}$ be a triple of points in $\pi \setminus \OO$ such that $\Delta = \{ \alpha_P,\alpha_Q , \alpha_R\}$ is a triangle which is not self-polar. Then the group $H = \langle \alpha_P,\alpha_R,\alpha_Q \rangle$ is isomorphic to $\PSL(2,q_0)$ or $\PGL(2,q_0)$, for some $q_0$ dividing $q$, or to $A_5$, $S_4$. 

\end{theorem}

\begin{proof}
    We first show that the only maximal subgroups $M$ of $\PGL(2,q)$ that can contain $H$ are of the form $\PGL(2,q')$ or $\PSL(2,q')$, for some $q'$ dividing $q$, or are isomorphic to $A_5$ or $S_4$. Then, either $H$ is one of the above groups, or the same reasoning can be applied to their maximal subgroups. We can continue to argue like this until we either find a $q_0$ that works, or we arrive at the conclusion that $H$ cannot be contained in any of the maximal subgroups of $\PGL(2,q)$, except potentially $\PSL(2,q)$, and that it must therefore be equal to $\PGL(2,q)$ or $\PSL(2,q)$.
    
    Suppose first that $\alpha_P,\alpha_Q$ and $\alpha_R$ are all in $\PSL(2,q)$. Then it is clear that $H$ is a subgroup of $\PSL(2,q)$. We use the classification of maximal subgroups of $\PSL(2,q)$ of Theorem \ref{thm:MaxSubPSL}. Recall that we see $\PSL(2,q)$ as a group of collineations of $\pi$ fixing $\OO$. As such, maximal subgroups of type (1),(2) and (3) in Theorem \ref{thm:MaxSubPSL} correspond to stabilizers of tangent, secant and exterior lines respectively. Suppose $l$ is a line of $\pi$ and that $H$ is a subgroup of the stabilizer $K$ of $l$ in $\PSL(2,q)$. Let $\alpha \in K$ be any involution. Then either the center of $\alpha$ is on $l$ or the axis of $\alpha$ is $l$. Since $\Delta$ is a triangle, the points  $P,Q$ and $R$ cannot all be on $l$. Moreover, any two involutions of $\PSL(2,q)$ having the same axis are equal, so at most one of $\alpha_P,\alpha_Q$ and $\alpha_R$ can have $l$ as an axis. This means that the only possible configuration for $\alpha_P,\alpha_Q$ and $\alpha_R$ to all be in $K$ is that two of the centers are in $l$ and one of the axis is $l$. Suppose thus that $l = (PQ)$ and that the axis of $\alpha_R$ is $l$. That is a direct contradiction with the fact that $\Delta$ is not self-polar. Therefore, $H$ cannot be a subgroup of a maximal subgroup of type (1),(2) or (3). 
    Subgroups of type (6),(7) and (8) corresponds to the groups $A_4$, $A_5$ and $S_4$. There are only $3$ involutions in $A_4$ and they all commute. So $H$ cannot be a subgroup of $A_4$. It is not hard to show that if $H$ is a subgroup of $A_5$ or $S_4$, it must actually be equal to $A_5$ or $S_4$.
    In conclusion, $H$ is either equal to $A_5$ or $S_4$, or is a subgroup of $\PGL(2,q')$ or $\PSL(2,q')$, for some $q'$ dividing $q$. Repeating this reasoning, we arrive to the desired conclusion.

    Suppose now that $\alpha_P \notin \PSL(2,q)$. Suppose first that $H$ contains $\PSL(2,q)$. Since $H$ also contains $\alpha_P$, $H$ must be the whole group $G = \PGL(2,q)$. So we can assume that $H$ does not contain $\PSL(2,q)$. We can then use the exact same arguments as for $\PSL(2,q)$, using Theorem \ref{thm:MaxSubPGL} this time, to get to the wanted results.
\end{proof}

We can now prove our second main theorem, that guarantees the existence of rank $3$ hypertopes with automorphism group $\PGL(2,q)$ for any values of $q$ in odd characteristic.
\begin{corollary}\label{coro:PGLnonLinear}
    For any $q = p^n$ with $p$ and odd prime, there exists an hypertope of rank $3$ whose diagram is not linear and whose automorphism group is $\PGL(2,q)$.
\end{corollary}
\begin{proof}
    Let $\alpha_P$ be any involutions in $\PGL(2,q) - \PSL(2,q)$ and let $l$ be any non-tangent line containing $P$. Then, the stabilizer of $l$ in $\PGL(2,q)$ is a dihedral group of order $m$ where $m = 2(q+1)$ or $m =2(q-1)$. In any case, we can find $\alpha_Q$ such that $Q \in l, \alpha_Q \in \PGL(2,q) - \PSL(2,q)$ and $\alpha_P\alpha_Q$ has order $m/2$. Finally, we choose a third involution $\alpha_R \in \PGL(2,q) - \PSL(2,q)$ such that $R \notin l$ and such that the axis of $\alpha_R$ does not contain neither $P$ nor $Q$. The last conditions guarantees that there is no pair of commuting involutions among the three involutions we chose, so that the diagram will indeed not be linear. The triangle $\Delta = \{\alpha_P,\alpha_Q,\alpha_R \}$ is then strongly not self-polar by \ref{prop:notPSL}. By Theorem \ref{thm:generationGeneral}, the group $H$ generated by the three involutions is isomorphic to $\PGL(2,q_0)$ or $\PSL(2,q_0)$, for some $q_0$ (note that $A_5$ is isomorphic to $\PSL(2,5)$ and $S_4$ to $\PGL(2,3)$). But $\PGL(2,q_0)$ does not contain elements of order $m/2$ if $q_0 < q$. Hence $q_0 = q$ and $H = \PGL(2,q)$. We conclude using Theorem \ref{thm:main}.
\end{proof}

On the other extreme, we can also find conditions to guarantee that the group $H$ generated by the three involutions we choose is isomorphic to $\PGL(2,p)$ or $\PSL(2,p)$.

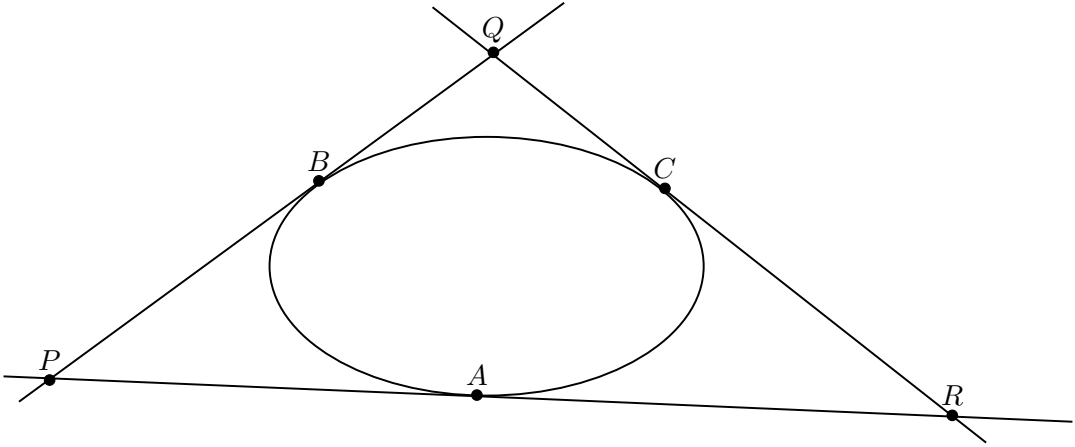
\begin{figure}
    \centering
    \begin{tikzpicture}
    \node[anchor=south west,inner sep=0] (image) at (0,0) {\includegraphics[width=0.9\textwidth]{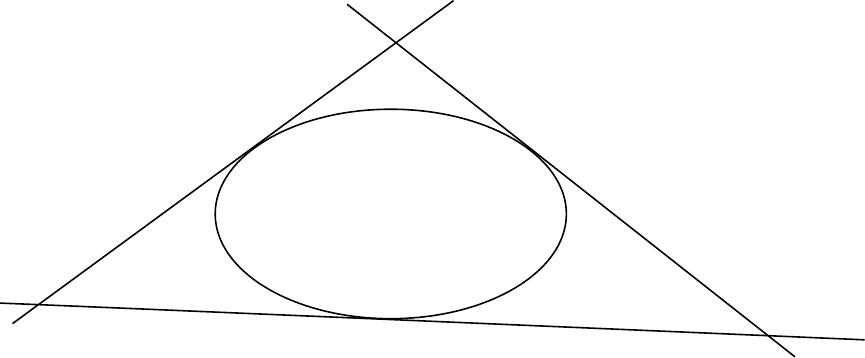}};
    \filldraw[black] (0.62,0.85) circle (2pt)  node[anchor=south]{$P$};
    \filldraw[black] (12.62,0.38) circle (2pt)  node[anchor=south]{$R$};
    \filldraw[black] (6.52,5.21) circle (2pt)  node[anchor=south]{$Q$};
    \filldraw[black] (4.2,3.5) circle (2pt)  node[anchor=south]{$B$};
    \filldraw[black] (8.8,3.4) circle (2pt)  node[anchor=south]{$C$};
    \filldraw[black] (6.3,0.65) circle (2pt)  node[anchor=south]{$A$};
    
    \end{tikzpicture}
    \caption{A tangent triangle with vertices $P,Q$ and $R$. The intersection points with the conics $\OO$ are labeled $A,B$ and $C$. }
    \label{fig:circumscribed}
\end{figure}

\begin{theorem}\label{thm:generationTangent}
    Let $\Delta = \{\alpha_P,\alpha_Q,\alpha_R \}$ be a triangle tangent to $\OO$, then $H = \langle \alpha_P,\alpha_Q,\alpha_R \rangle$ is isomorphic to $\PSL(2,p)$ if $p = 1 \pmod 4$ or to $\PGL(2,p)$ if $p = 3 \pmod 4$. 
\end{theorem}
\begin{proof}
     Suppose that $p = 1 \pmod 4$. We first show that $H \subset \PSL(2,p)$.
     As always, the group $G = \PGL(2,q)$ is acting on $\pi = \PG(2,q)$ as the stabilizer of a conic $\OO$. Under that action, points of $\pi$ that do not lie in $\OO$ correspond bijectively to the center of the involutions of $G$. The group $G$ contains many subgroups isomorphic to $\PGL(2,p)$, which are all conjugate. For each of these subgroups, the set of the centers of their involutions corresponds to the points of a projective plane $\pi_p = \PG(2,p) \subset \pi$. The three involutions $\alpha_P,\alpha_Q$ and $\alpha_R$ are then contained in a subgroup isomorphic to $\PGL(2,p)$ if and only if their centers are in such a subplane $\pi_p \subset \pi$. Let one such subplane $\pi_p$ be fixed once and for all. Let $A,B,C$ be the three points $(PQ) \cap \OO, (QR) \cap \OO$ and $(PR) \cap \OO$, respectively, as illustrated in Figure \ref{fig:circumscribed}. Since $G$ acts $3$-transitively on $\OO$, we can suppose that $A,B,C$ belongs to $\pi_p$. Then, the tangent lines to $\OO$ at $A,B$ and $C$ are lines of $\pi_p$, and so their intersections are in $\pi_p$. But these intersections are precisely $P,Q$ and $R$.

    This shows that $H \subset \PGL(2,p)$. The involutions $\alpha_P,\alpha_Q$ and $\alpha_R$ all have 2 fixed points on $\OO \cap \pi_p$ and $p = 1 \pmod 4$, so Lemma \ref{lem:fixedpointsonO} allows us to conclude that $H \subset \PSL(2,p)$.

    We now show that $\PSL(2,p) \subset H$. We can use Theorem \ref{thm:generationGeneral}, as triangles tangent to $\OO$ are never self-polar. Clearly, the only subgroup of type $\PGL(2,q_0)$ or $\PSL(2,q_0)$, for some $q_0$ dividing $q$, contained in $\PSL(2,p)$ is $\PSL(2,p)$ itself. Finally, unless $p = 3,5$, the group $H$ cannot be isomorphic to $A_5$ or $S_4$. Indeed, the order of $\alpha_P\alpha_Q$ is always $p$. In the case of $p=3$ or $5$, the group $A_5$ is isomorphic to $\PSL(2,5)$ and $A_4$ is isomorphic to $\PGL(2,3)$. So the theorem holds in every case.

    The proof for $p = 3 \pmod 4$ is mostly identical.

\end{proof}

We obtain the following as a direct corollary

\begin{corollary}\label{coro:tangent}
    Let $p = 3 \mod 4$ and let $\Delta = \{\alpha_P,\alpha_Q,\alpha_R \}$ be a triangle tangent to $\OO$. Then $\Gamma = \Gamma(\PGL(2,p),(H_0,H_1,H_2))$ is a rank $3$ hypertope. Moreover, we have that $\Cor(\Gamma)/\Aut(\Gamma) = S_3$.
\end{corollary}

\begin{proof}
    Theorem \ref{thm:main} together with Theorem \ref{thm:generationTangent} yield that $\Gamma$ is a rank $3$ hypertope. Referring to Figure \ref{fig:circumscribed} for notations, we see that, since $\PGL(2,p)$ acts $3$-transitively on $\OO$, we can permute as we want the triple of points $\{A,B,C\}$. Any such permutation will also permute the triple $\{P,Q,R\}$ and thus induce a correlation of $\Gamma$.
\end{proof}

We conclude this article by a corollary that states that, when restricted to an appropriate subplane, fields automorphisms actually behave as projectivities.

Let $\pi = \PG(2,q^3)$ be a projective plane of order $q^3$ where $q = p^n$ for some odd prime $p$ and some integer $n$. Let $[X_0:X_1:X_2]$ be homogeneous coordinates for $\pi$ and let $\tau \colon \FF_{q^3} \to \FF_{q^3}$ be the automorphism of $\FF_{q^3}$ sending $x$ to $x^q$ for all $x \in \FF_{q^3}$, so that $\tau$ has order $3$. By abuse of notation, we will also consider $\tau$ as being a collineation of $\pi$ and an automorphism of $G = \PGL(2,q^3)$.

We will say that a triangle $\Delta = \{\alpha_P,\alpha_Q,\alpha_R \}$ is a $\tau-$triangle if $Q = \tau(P)$ and $R = \tau^2(P)$.
\begin{corollary}\label{coro:coliToproj}
    Let $\tau$ be the collineation of order $3$ of $\pi = \PG(2,p^3)$ coming from the Frobenius automorphism. Then there exists a projective plane $\pi_p = \PG(2,p)$ embedded in $\pi$ such that the restriction of $\tau$ to $\pi_p$ is a projectivity.
\end{corollary}

\begin{proof}
    Let $\Delta$ be a $\tau$-triangle tangent to $\OO$. This clearly exists, as it suffice to take $A,B$ and $C$ in Figure \ref{fig:circumscribed} such that $B = \tau(A)$ and $C = \tau^2(A)$. Then, by Theorem \ref{thm:generationTangent}, we know that the group $H$ generated by the three involutions of $\Delta$ is isomorphic to $\PGL(2,p)$ or $\PSL(2,p)$ and that $A,B$ and $C$ belong to a subplane $\pi_p = \PG(2,p) \subset \pi$ (see proof of Theorem \ref{thm:generationTangent}). Note that, under this settings, $\tau$ induces an automorphism of $H$. As $\PGL(2,p)$ acts sharply $3-$transitively on $\OO_p = \OO \cap \pi_p$, we see that there exists a unique element $g$ of $\PGL(2,p)$ that permutes $A,B$ and $C$ cyclically. Conjugation by $g$ must then be equal to the action of $\tau$ on $H$, since they both act in the same way on the generators of $H$. Therefore, $\tau$ acts as an inner conjugation on $H$, and thus as a projectivity on the associated projective plane $\pi_p$.
    
\end{proof}

\section{Further work}

In this article, we have showed, in Theorem \ref{thm:main}, that a triangle $\Delta = \{ \alpha_P,\alpha_Q,\alpha_R \}$ that is strongly non self-polar can always be used to construct a regular hypertope $\Gamma$. Moreover, by Proposition \ref{prop:notPSL}, we know that we can guarantee that $\Delta$ is strongly non self-polar by choosing involutions $\alpha_P,\alpha_Q$ and $\alpha_R$ that are not contained in $\PSL(2,q)$. A natural follow up is to find condition for a triangle $\Delta$ of involutions all contained in $\PSL(2,q)$ to be strongly non self-polar.

\begin{question}
    How to choose three involutions $\alpha_P,\alpha_Q$ and $\alpha_R$ in $\PSL(2,q)$ in order to guarantee that $\Delta = \{ \alpha_P,\alpha_Q,\alpha_R \}$ is strongly non self-polar?
\end{question}

Computational data obtained with {\sc Magma} suggest that not only strongly non self-polar triangle in $\PSL(2,q)$ do exist for every value of $q$, but also that almost all triangle $\Delta = \{ \alpha_P,\alpha_Q,\alpha_R \}$ such that $\PSL(2,q) = \langle \alpha_P,\alpha_Q,\alpha_R \rangle$ are strongly non self-polar.

There has been recent interest in geometries that admit trialities but no dualities (\cite{leemans2023flag}, \cite{leemans2022incidence}). Computational evidence again suggests that these geometries exist for $\PGL(2,q)$ and $\PSL(2,q)$ whenever $q$ is a cube. Let $q$ be a cube and let $\tau$ be the Frobenius automorphism of order $3$, seen as a collineation of $\pi = \PG(2,q)$.

\begin{question}
    Can we find an involution $\alpha_P$ in $\PGL(2,q)$ such that the $\tau-$triangle $\Delta_\tau = \{ \alpha_P,\alpha_{\tau(P)}, \alpha_{\tau^2(P)}\}$ is strongly non self-polar and show that the associated regular hypertope $\Gamma$ has no dualities?
\end{question}
Remark that choosing $\alpha_P$ not contained in $\PSL(2,q)$ guarantees that $\Delta_\tau$ will be strongly non self-polar. To guarantee that the three involutions generate $\PGL(2,q)$, it would then suffice to show that $\alpha_P$ can be chosen such that $\Delta_\tau$ is proper.

\bibliographystyle{abbrv} 
\bibliography{refs}

\begin{thebibliography}{10}

\bibitem{baer1946projectivities}
R.~Baer.
\newblock Projectivities with fixed points on every line of the plane.
\newblock 1946.

\bibitem{buek}
F.~Buekenhout.
\newblock Diagrams for geometries and groups.
\newblock {\em J. Combin. Theory Ser. A}, 27(2):121--151, 1979.

\bibitem{buekenhout2013diagram}
F.~Buekenhout and A.~M. Cohen.
\newblock {\em Diagram geometry: related to classical groups and buildings},
  volume~57.
\newblock Springer Science \& Business Media, 2013.

\bibitem{3-design}
P.~J. Cameron, G.~R. Omidi, and B.~Tayfeh-Rezaie.
\newblock 3-designs from {${\rm PGL}(2,q)$}.
\newblock {\em Electron. J. Combin.}, 13(1):Research Paper 50, 11, 2006.

\bibitem{cherkassoff1993groups}
M.~Cherkassoff and D.~Sjerve.
\newblock On groups generated by three involutions, two of which commute.
\newblock In {\em The Hilton Symposium}, volume~6, pages 169--185, 1993.

\bibitem{RegularHypermapsPGL}
M.~Conder, P.~Poto\v{c}nik, and J.~\v{S}ir\'{a}\v{n}.
\newblock Regular hypermaps over projective linear groups.
\newblock {\em J. Aust. Math. Soc.}, 85(2):155--175, 2008.

\bibitem{CgroupPGL}
T.~Connor, S.~Jambor, and D.~Leemans.
\newblock C-groups of {${\rm PSL}(2,q)$} and {${\rm PGL}(2,q)$}.
\newblock {\em J. Algebra}, 427:455--466, 2015.

\bibitem{CgroupSuzuki}
T.~Connor and D.~Leemans.
\newblock C-groups of {S}uzuki type.
\newblock {\em J. Algebraic Combin.}, 42(3):849--860, 2015.

\bibitem{coxeter2003projective}
H.~S.~M. Coxeter.
\newblock {\em Projective geometry}.
\newblock Springer Science \& Business Media, 2003.

\bibitem{Dickson1958LinearGW}
L.~E. Dickson.
\newblock Linear groups, with an exposition of the galois field theory.
\newblock 1958.

\bibitem{HighlySymmetricHypertopes}
M.~E. Fernandes, D.~Leemans, and A.~I. Weiss.
\newblock Highly symmetric hypertopes.
\newblock {\em Aequationes Math.}, 90(5):1045--1067, 2016.

\bibitem{giudici2007maximal}
M.~Giudici.
\newblock Maximal subgroups of almost simple groups with socle $psl (2, q) $.
\newblock {\em arXiv preprint math/0703685}, 2007.

\bibitem{hirschfeld1998projective}
J.~W.~P. Hirschfeld.
\newblock {\em Projective geometries over finite fields}.
\newblock Oxford University Press, 1998.

\bibitem{ChiralPolyPGL}
D.~Leemans, J.~Moerenhout, and E.~O'Reilly-Regueiro.
\newblock Projective linear groups as automorphism groups of chiral polytopes.
\newblock {\em J. Geom.}, 108(2):675--702, 2017.

\bibitem{Leemans2009Polytope}
D.~Leemans and E.~Schulte.
\newblock Polytopes with groups of type {${\rm PGL}_2(q)$}.
\newblock {\em Ars Math. Contemp.}, 2(2):163--171, 2009.

\bibitem{leemans2022incidence}
D.~Leemans and K.~Stokes.
\newblock Incidence geometries with trialities coming from maps with {Wilson}
  trialities.
\newblock {\em Innov. Incidence Geom.}, 20(2--3):325--340, 2023.

\bibitem{leemans2023flag}
D.~Leemans, K.~Stokes, and P.~Tranchida.
\newblock Flag transitive geometries with trialities and no dualities coming
  from suzuki groups.
\newblock {\em arXiv preprint arXiv:2311.13522}, 2023.

\bibitem{mcmullen2002abstract}
P.~McMullen and E.~Schulte.
\newblock {\em Abstract regular polytopes}, volume~92.
\newblock Cambridge University Press, 2002.

\bibitem{moore1903subgroups}
E.~H. Moore.
\newblock {\em The subgroups of the generalized finite modular group},
  volume~9.
\newblock University of Chicago Press, 1903.

\bibitem{Tits1957}
J.~Tits.
\newblock Sur les analogues alg\'{e}briques des groupes semi-simples complexes.
\newblock In {\em Colloque d'alg\`ebre sup\'{e}rieure, tenu \`a {B}ruxelles du
  19 au 22 d\'{e}cembre 1956}, Centre Belge de Recherches Math\'{e}matiques,
  pages 261--289. \'{E}tablissements Ceuterick, Louvain, 1957.

\bibitem{Tits1974}
J.~Tits.
\newblock {\em Buildings of spherical type and finite {BN}-pairs}.
\newblock Lecture Notes in Mathematics, Vol. 386. Springer-Verlag, Berlin-New
  York, 1974.

\bibitem{wiman1899bestimmung}
A.~Wiman.
\newblock {\em Bestimmung aller Untergruppen: doppelt unendlichen Reihe von
  einfachen Gruppen}.
\newblock PA Norstedt \& S{\"o}ner, 1899.

\end{thebibliography}
\end{document}